%% file: MINE_equitability.tex
\documentclass[11pt]{article}
\usepackage[margin=1.13in]{geometry}

\date{}

\usepackage[numbers,sort&compress]{natbib}
\usepackage{graphicx}
\usepackage[font=small,labelfont=bf]{caption}
\usepackage{amsmath, amssymb, amsthm}
\usepackage[colorlinks=true, urlcolor=blue, citecolor=blue, linkcolor=blue, filecolor=blue]{hyperref}
\usepackage[textsize=footnotesize]{todonotes}

\newcommand{\pathToCommon}{.}
\newcommand{\pathToCommonFigs}{.}
\newcommand{\pathToFigures}{.}
\input{\pathToCommon/commands.tex}

\makeatletter
\renewcommand*{\@fnsymbol}[1]{\ensuremath{\ifcase#1\or 1\or *\or \dagger\or 2\or 3\or 4\or **\or 5\or \mathsection\or \mathparagraph\or \|\or \ddagger\or \dagger\dagger
   \or \ddagger\ddagger \else\@ctrerr\fi}}
\makeatother

\title{Equitability, interval estimation, and statistical power}
\author{
Yakir A.\ Reshef\footnote{School of Engineering and Applied Sciences, Harvard University.} \footnote{Co-first author.} \footnote{To whom correspondence should be addressed. Email: \url{yakir@seas.harvard.edu}}
\and
David N. Reshef\footnote{Department of Computer Science, Massachusetts Institute of Technology.} \footnotemark[2]
\and
Pardis C. Sabeti\footnote{Department of Organismic and Evolutionary Biology, Harvard University.} \footnote{Broad Institute of MIT and Harvard.} \footnote{Co-last author.}
\and
Michael M. Mitzenmacher\footnotemark[1] \footnotemark[7]
}

\begin{document}

\maketitle
\begin{abstract}
As data sets grow in dimensionality, non-parametric measures of dependence have seen increasing use in data exploration due to their ability to identify non-trivial relationships of all kinds. One common use of these tools is to test a null hypothesis of statistical independence on all variable pairs in a data set. However, because this approach attempts to identify any non-trivial relationship no matter how weak, it is prone to identifying so many relationships --- even after correction for multiple hypothesis testing --- that meaningful follow-up of each one is impossible. What is needed is a way of identifying a smaller set of ``strongest'' relationships of all kinds that merit detailed further analysis.

Here we formally present and characterize {\em equitability}, a property of measures of dependence that aims to overcome this challenge. Notionally, an equitable statistic is a statistic that, given some measure of noise, assigns similar scores to equally noisy relationships of different types (e.g., linear, exponential, etc.) \cite{MINE}. We begin by formalizing this idea via a new object called the interpretable interval, which functions as an interval estimate of the amount of noise in a relationship of unknown type. We define an equitable statistic as one with small interpretable intervals.

We then draw on the equivalence of interval estimation and hypothesis testing to show that under moderate assumptions an equitable statistic is one that yields well powered tests for distinguishing not only between trivial and non-trivial relationships of all kinds but also between non-trivial relationships of different strengths, regardless of relationship type. This means that equitability allows us to specify a threshold relationship strength $x_0$ below which we are uninterested, and to search a data set for relationships of all kinds with strength greater than $x_0$. Thus, equitability can be thought of as a strengthening of power against independence that enables fruitful analysis of data sets with a small number of strong, interesting relationships and a large number of weaker, less interesting ones. We conclude with a demonstration of how our two equivalent characterizations of equitability can be used to evaluate the equitability of a statistic in practice.
\end{abstract}

\section{Introduction}

Suppose we have a data set that we would like to explore to find pairwise associations of interest. A commonly taken approach that makes minimal assumptions about the structure in the data is to compute a measure of dependence, i.e., a statistic whose population value is non-zero exactly in cases of statistical dependence, on many candidate pairs of variables. The score of each variable pair can be evaluated against a null hypothesis of statistical independence, and variable pairs with significant scores can be kept for follow-up \cite{storey2003statistical,emilsson2008genetics}. When faced with this task, there is a wealth of measures of dependence from which to choose, each with a different set of properties~\cite{reshef2015estimating, szekely2007measuring, Kraskov, breiman1985estimating, hoeffding1948non, heller2013consistent, jiang2014non, gretton2005measuring, gretton2012kernel, lopez2013randomized}.

While this approach works well in some settings, it is unsuitable in many others due to the size of modern data sets. In particular, as data sets grow in dimensionality, the above approach often results in lists of significant relationships that are too large to allow for meaningful follow-up of every identified relationship. For example, in the gene expression data set analyzed in \cite{heller2014consistent}, several measures of dependence reliably identified thousands of significant relationships amounting to between $65$ and $75$ percent of the variable pairs in the data set. Given the extensive manual effort that is usually necessary to better understand each of these ``hits'', further characterizing all of them is impractical.

A tempting way to deal with this challenge is to rank all the variable pairs in a data set according to the test statistic used (or according to p-value) and to examine only a small number of pairs with the most extreme values. However, this is a poor idea because, while a measure of dependence guarantees non-zero scores to dependent variable pairs, the magnitude of these non-zero scores can depend heavily on the type of dependence in question, thereby skewing the top of the list toward certain types of relationships over others. For example, if some measure of dependence $\varphi$ systematically assigns higher scores to, say, linear relationships than to sinusoidal relationships, then using $\varphi$ to rank variable pairs in a large data set could cause noisy linear relationships in the data set to crowd out strong sinusoidal relationships from the top of the list. The natural result would be that the human examining the top-ranked relationships would never see the sinusoidal relationships, and they would not be discovered.

The consistency guarantee of measures of dependence is therefore not strong enough to solve the data exploration problem posed here. What is needed is a way not just to identify as many relationships of different kinds as possible in a data set, but also to identify a small number of strongest relationships of different kinds.

Here we formally present and characterize {\em equitability}, a framework for meeting this goal. In previous work, equitability was informally introduced as follows: an equitable measure of dependence is one that, given some measure of noise, assigns similar scores to equally noisy relationships, regardless of relationship type \cite{MINE}. In this paper, we formalize this notion in the language of estimation theory and tie it to the theory of hypothesis testing.

Specifically, we define an object called the interpretable interval that functions as an interval estimate of the strength of a relationship of unknown type. That is, given a set $\Q$ of standard relationships on which we have defined a measure $\Phi$ of relationship strength, the interpretable interval is a range of values that act as good estimates of the true relationship strength $\Phi$ of a distribution, assuming it belongs to $\Q$. In the same way that a good estimator has narrow confidence intervals, an equitable statistic is one that has narrow interpretable intervals. As we explain, this property can be viewed as a natural generalization of one of the ``fundamental properties'' described by Renyi in his framework for measures of dependence~\cite{renyi1959measures}.

We then draw a connection between equitability and statistical power using the equivalence between interval estimation and hypothesis testing. This connection shows that whereas typical measures of dependence are analyzed in terms of power to distinguish non-trivial associations from statistical independence, under moderate assumptions an equitable statistic is one that can distinguish finely between relationships of two different strengths that may both be non-trivial, regardless of the types of the two relationships in question. This result gives us a new way to understand equitability as a natural strengthening of the requirement of power against independence in which we ask that our statistic be useful not just for detecting deviations of different types from independence but also for distinguishing strong relationships from weak relationships regardless of relationship type.

Finally, motivated by the connection between equitability and power, we define a new property, {\em detection threshold}, which, at some fixed sample size, is the minimal relationship strength $x$ such that a statistic's corresponding independence test has a certain minimal power on relationships of all kinds with strength at least $x$. We show that low detection threshold is strictly weaker than high equitability in that high equitability implies it but the converse does not hold. Therefore, when equitability is too much to ask, low detection threshold on a broad set of relationships with respect to an interesting measure of relationship strength may be a reasonable surrogate goal.

Throughout this paper, we give concrete examples of how our formalism relates to the analysis of equitability in practice. Indeed, the purpose of the theoretical framework provided here is to allow for such practical analyses, and so we close with a demonstration of an empirical analysis of the equitability of several popular measures of dependence.

This paper is accompanied by two companion papers. The first \cite{reshef2015estimating} introduces two new statistics that aim for good equitability on functional relationships and good power against statistical independence, respectively. The second \cite{reshef2015comparisons} conducts a comprehensive empirical analysis of the equitability and power against independence of both of these new methods as well as several other leading measures of dependence.

The results we present here, in addition to contributing to a better understanding of equitability, also provide an organizing framework in which to consolidate some of the recent discussion around equitability. For instance, our formalization of equitability is sufficiently general to accommodate several of variants that have arisen in the literature. This allows us to precisely discuss the definition given by Kinney and Atwal \cite{kinney2014equitability, reshef2014comment} of what, in our theoretical framework, corresponds to perfect equitability. In particular, our framework allows us to explain the limitations of an impossibility result presented by Kinney and Atwal about perfect equitability. Additionally, our framework and the connection it provides to statistical power also allows us to crystallize and address the concerns about the power against independence of equitable methods raised by Simon and Tibshirani \cite{simon2012comment}. (However, empirical questions concerning the performance of the maximal information coefficient and related statistics are deferred to the companion papers \cite{reshef2015comparisons, reshef2015estimating}.)

We conclude with a discussion of what situations benefit from using equitability as a desideratum for data analysis. It is our hope that the theoretical results in this paper will provide a foundation for further work not only on equitability and methods for achieving equitability, but also on other possible expansions of our goals for measures of dependence in the setting of data exploration or other related settings.

\section{Equitability}
\label{sec:equitability}
Equitability has been described informally by the authors as the ability of a statistic to ``give similar scores to equally noisy relationships of different types'' \cite{MINE}. Though useful, this informal definition is imprecise in that it does not specify what is meant by ``noisy'' or ``similar'', and does not specify for which relationships the stated property should hold. In this section we provide the formalism necessary to discuss equitability more rigorously.

To do this, we fix a statistic $\hvphi$ (presumed to be a measure of dependence), a measure of relationship strength $\Phi$ called the {\em property of interest}, and a set $\Q$ of {\em standard relationships} on which $\Phi$ is defined. The idea is that $\Q$ contains relationships of many different types, and for any distribution $\mcZ \in \Q$, $\Phi(\mcZ)$ is the way we would ideally quantify the strength of $\mcZ$ if we had knowledge of the distribution $\mcZ$. Our goal is then, given a sample $Z$ of size $n$ from $\mcZ$, to use $\hvphi(Z)$ to draw inferences about $\Phi(\mcZ)$.

Our general approach is to construct a set of intervals, the {\em interpretable intervals} of $\hvphi$ with respect to $\Phi$, by inverting a certain set of hypothesis tests. We show that these intervals can be used to turn $\hvphi(Z)$ into an interval estimate of $\Phi(\mcZ)$, and we call the statistic $\hvphi$ equitable if its interpretable intervals are small, i.e., if it yields narrow interval estimates of $\Phi(\mcZ)$.

After constructing the interpretable intervals of $\hvphi$ with respect to $\Phi$, we demonstrate how our vocabulary can be used to define a few different concrete instantiations of the concept of equitability. We do this by using our framework to state several of the notions of- and results about equitability that have appeared in the literature, and discussing the relationships among them. Following this, we provide a short schematic illustration of how the definitions we provide would be used to quantitatively evaluate the equitability of a statistic in practice, and a discussion of how equitability is related to measurement of effect size more generally.

In what follows, we keep our exposition generic in order to accommodate variations \--- both existing and potential \--- on the concepts defined here. However, as a motivating example, we often return to the setting of \cite{MINE}, in which $\hvphi$ is a statistic like the maximal information coefficient $\MICestE$, $\Q$ is a set of noisy functional relationships, and $\Phi$ is the coefficient of determination ($R^2$) with respect to the generating function. In this setting, the equitability of $\MICestE$ corresponds to its utility for constructing narrow interval estimates of the $R^2$ of a relationship that is in $\Q$ but whose specific functional form is unknown.

\subsection{Interpretable intervals}
Let $\hvphi$ be a statistic taking values in $[0,1]$, let $\Q$ be a set of distributions, and let $\Phi : \Q \rightarrow [0,1]$ be some measure of relationship strength. As mentioned previously, we refer to $\Q$ as the set of standard relationships and to $\Phi$ as the property of interest. To construct the interpretable intervals of $\hvphi$ with respect to $\Phi$, we must first ask how much $\hvphi$ can vary when evaluated on a sample from some $\mcZ \in \Q$ with $\Phi(\mcZ) = x$. The definition below gives us a way to measure this. (In this definition and in definitions in the rest of this paper, we implicitly assume a fixed sample size of $n$.)

\begin{definition}[Reliability of a statistic]
Let $\hvphi$ be a statistic taking values in $[0,1]$, and let $x, \alpha \in [0,1]$. The $\alpha$-reliable interval of $\hvphi$ at $x$, denoted by $\reliablestat{\alpha}{\hvphi}{x}$, is the smallest closed interval $A$ with the property that, for all $\mcZ \in \Q$ with $\Phi(\mcZ) = x$, we have
\[
\Pr{\hvphi(Z) < \min A} < \alpha/2 \quad \mbox{and} \quad \Pr{\hvphi(Z) > \max A} < \alpha/2
\]
where $Z$ is a sample of size $n$ from $\mcZ$.

The statistic $\hvphi$ is {\em $1/d$-reliable} with respect to $\Phi$ on $\Q$ at $x$ with probability $1-\alpha$ if and only if the diameter of $\reliablestat{\alpha}{\hvphi}{x}$ is at most $d$.
\end{definition}

See Figure~\ref{fig:reliabilityInterpretability}a for an illustration. The reliable interval at $x$ is an acceptance region of a size-$\alpha$ test of the null hypothesis $H_0:\Phi(\mcZ) = x$. If there is only one $\mcZ$ satisfying $\Phi(\mcZ) = x$, this amounts to a central interval of the sampling distribution of $\hvphi$ on $\mcZ$. If there is more than one such $\mcZ$, the reliable interval expands to include the relevant central intervals of the sampling distributions of $\hvphi$ on all the distributions $\mcZ$ in question. For example, when $\Q$ is a set of noisy functional relationships with several different function types and $\Phi$ is $R^2$, the reliable interval at $x$ is the smallest interval $A$ such that for any functional relationship $\mcZ \in \Q$ with $R^2(\mcZ) = x$, $\hvphi(Z)$ falls in $A$ with high probability over the sample $Z$ of size $n$ from $\mcZ$.

Because the reliable interval $\reliablestat{\alpha}{\hvphi}{x}$ can be viewed as the acceptance region of a level-$\alpha$ test of $H_0 : \Phi(\mcZ) = x$, the equivalence between hypothesis tests and confidence intervals yields interval estimates of $\Phi$ in terms of $\reliablestat{\alpha}{\hvphi}{x}$. These intervals are the interpretable intervals, defined below.
\begin{definition}[Interpretability of a statistic]
Let $\hvphi$ be a statistic taking values in $[0,1]$, and let $y, \alpha \in [0,1]$. The $\alpha$-interpretable interval of $\hvphi$ at $y$, denoted by $\interpretablestat{\alpha}{\hvphi}{y}$, is the smallest closed interval containing the set
\[ \left\{ x \in [0,1] : y \in \reliablestat{\alpha}{\hvphi}{x} \right\} .\]

The statistic $\hvphi$ is {\em $1/d$-interpretable} with respect to $\Phi$ on $\Q$ at $y$ with confidence $1-\alpha$ if and only if the diameter of $\interpretablestat{\alpha}{\hvphi}{y}$ is at most $d$.
\end{definition}
See Figure~\ref{fig:reliabilityInterpretability}a for an illustration. The correspondence between hypothesis tests and interval estimates~\cite{casella2002statistical} gives us the following guarantee about the coverage probability of the interpretable interval, whose proof we omit.
\begin{prop}
\label{prop:interval_estimates}
Let $\hvphi$ be a statistic taking values in $[0,1]$, and let $\alpha \in [0,1]$. For all $x \in [0,1]$ and for all $\mcZ \in \Q$,
\[
\Pr{\Phi(\mcZ) \in \interpretablestat{\alpha}{\hvphi}{\hvphi(Z)}} \geq 1-\alpha
\]
where $Z$ is a sample of size $n$ from $\mcZ$.
\end{prop}

The definitions just presented have natural non-stochastic counterparts in the large-sample limit that we summarize below.
\begin{definition}[Reliability and interpretability in the large-sample limit]
Let $\varphi : \Q \rightarrow [0,1]$ be a function of distributions. For $x \in [0,1]$, the smallest closed interval containing the set $\varphi(\Phi^{-1}(\{x\}))$ is called the {\em reliable interval} of $\varphi$ at $x$ and is denoted by $\reliable{\varphi}{x}$. For $y \in [0,1]$, the smallest closed interval containing the set $\{ x : y \in \reliable{\varphi}{x} \}$ is called the {\em interpretable interval} of $\varphi$ at $y$ and is denoted by $\interpretable{\varphi}{y}$.
\end{definition}
See Figure~\ref{fig:reliabilityInterpretability}b for an illustration.

\begin{figure}[h]
	\centering
	\begin{tabular}[c]{cc}
        \includegraphics[clip=true, trim = 0.35in 0.25in 0in 0.625in, width=0.4\textwidth]{\pathToCommonFigs/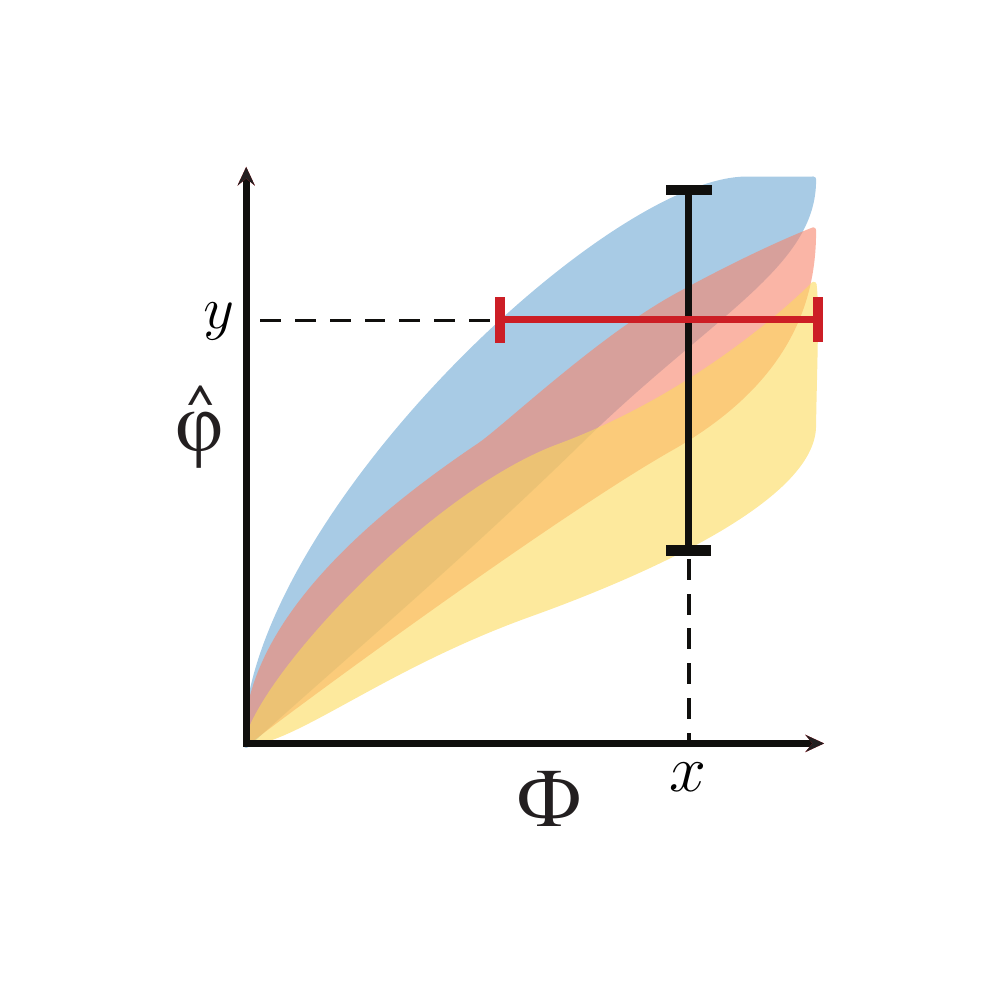}
                  &
        \includegraphics[clip=true, trim = 0.35in 0.25in 0in 0.625in, width=0.4\textwidth]{\pathToCommonFigs/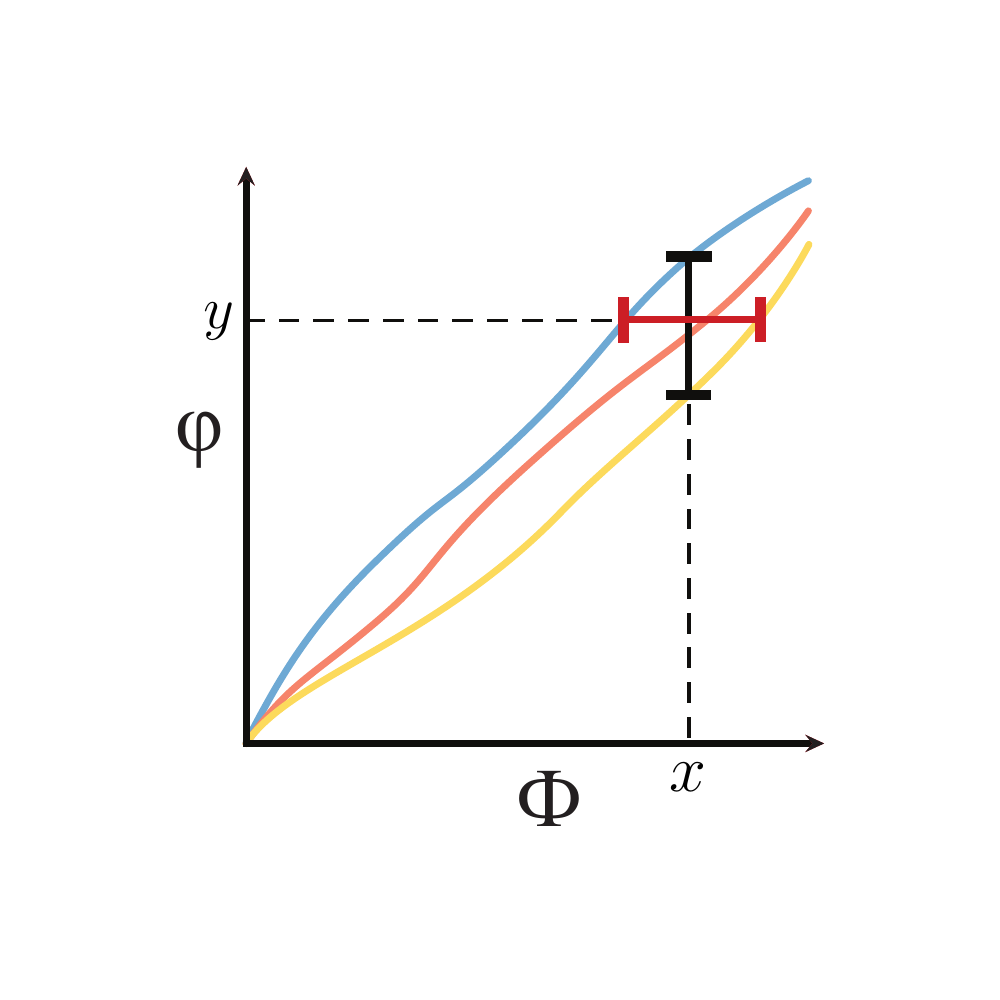} \\
       (a) & (b) \\
   \end{tabular}
    
  \caption{A schematic illustration of reliable and interpretable intervals. In both figure parts, $\Q$ consists of noisy relationships of three different types depicted in the three different colors.
\figpart{a} The relationship between a statistic $\hvphi$ and $\Phi$ on $\Q$ at a finite sample size. The bottom and top boundaries of each shaded region indicate the $(\alpha/2) 100\%$ and $(1-\alpha/2)\cdot 100\%$ percentiles of the sampling distribution of $\hvphi$ for each relationship type at various values of $\Phi$. The vertical interval (in black) is the reliable interval $\reliablestat{\alpha}{\hvphi}{x}$, and the horizontal interval (in red) is the interpretable interval $\interpretablestat{\alpha}{\hvphi}{y}$.
\figpart{b} In the large-sample limit, we replace $\hvphi$ with a population quantity $\varphi$. The vertical interval (in black) is the reliable interval $\reliable{\varphi}{x}$, and the horizontal interval (in red) is the interpretable interval $\interpretable{\varphi}{y}$.}
  \label{fig:reliabilityInterpretability}
\end{figure}

\subsection{Defining equitability}
Proposition~\ref{prop:interval_estimates} implies that if the interpretable intervals of $\hvphi$ with respect to $\Phi$ are small then $\hvphi$ will give good interval estimates of $\Phi$. There are many ways to summarize whether the interpretable intervals of $\hvphi$ are small; we focus here on two simple ones.

\begin{definition}
The worst-case $\alpha$-reliability (resp. $\alpha$-interpretability) of $\hvphi$ is $1/d$ if it is $1/d$-reliable (resp. interpretable) at all $x$ (resp. $y$) $\in [0,1]$. $\hvphi$ is said to be {\em worst-case $1/d$-reliable} (resp. {\em $1/d$-interpretable}) with probability (resp. confidence) $1-\alpha$.

The average-case $\alpha$-reliability (resp. $\alpha$-interpretability) of $\hvphi$ is $1/d$ if its reliability (resp. interpretability), averaged over all $x$ (resp. $y$) $\in [0,1]$, is at least $1/d$. $\hvphi$ is said to be {\em average-case $1/d$-reliable} (resp. {\em $1/d$-interpretable}) with probability (resp. confidence) $1-\alpha$.
\end{definition}
(One could imagine more fine-grained ways to summarize reliability/interpretability according to, for example, some prior over the distributions in $\Q$ that reflects a belief about the importance or prevalence of various types of relationships; for simplicity, we do not pursue this here.)

With this vocabulary, we can now define equitability: {\em average/worst-case equitability} is simply average/worst-case interpretability with respect to some $\Phi$ that reflects relationship strength. In this paper, we distinguish between interpretability in general and equitability specifically by using ``interpretability'' in general statements and ``equitability'' in contexts in which $\Phi$ is specifically considered as a measure of relationship strength. Also, we often use ``interpretability'' and ``equitability'' with no qualifier to mean worst-case interpretability/equitability.

The corresponding definitions of average/worst-case interpretability/reliability can be made for $\varphi$ in the large-sample limit as well. In that setting, it is possible that all the interpretable intervals of $\varphi$ with respect to $\Phi$ have size 0; that is, the value of $\varphi(\mcZ)$ uniquely determines the value of $\Phi(\mcZ)$. In this case, the worst-case reliability/interpretability of $\varphi$ is $\infty$, and $\varphi$ is said to be {\em perfectly reliable/interpretable}, or {\em perfectly equitable} depending on context.

Before continuing, let us build intuition by giving two examples of statistics that are perfectly interpretable in the large-sample limit. First, the mutual information \cite{Cover2006, csiszar2008axiomatic} is perfectly interpretable with respect to the correlation $\rho^2$ on the set $\Q$ of bivariate normal random variables. This is because for bivariate normals we have that $1-2^{-2I} = \rho^2$ \cite{linfoot1957informational}. Additionally, Theorem 6 of~\cite{szekely2009brownian} shows that for bivariate normals distance correlation is a deterministic function of $\rho^2$ as well. Therefore, distance correlation is also perfectly interpretable and perfectly reliable with respect to $\rho^2$ on the set of bivariate normals $\Q$.

The perfect interpretability with respect to $\rho^2$ on bivariate normals exhibited in both of these examples is in fact equivalent to one of the ``fundamental properties'' introduced by Renyi in his framework for thinking about ideal properties of measures of dependence \cite{renyi1959measures}. This property contains a compromise: it guarantees interpretability that on the one hand is perfect, but on the other hand applies only on a relatively small set of standard relationships. One goal of equitability is to give us the tools to relax the ``perfect'' requirement in exchange for the ability to make $\Q$ a much larger set, e.g., a set of noisy functional relationships. Thus, equitability can be viewed as a generalization of Renyi's requirement that allows for a tradeoff between the precision with which our statistic tells us about $\Phi$ and the set $\Q$ on which it does so.

\subsection{Examples of- and results about equitability}
We now give examples, using the vocabulary developed here, of some concrete instantiations of- and results about equitability. Our focus here is on functional relationships, as defined below.

\begin{definition}
A random variable distributed over $\R^2$ is called a {\em noisy functional relationship} if and only if it can be written in the form $(X + \ep, f(X) + \ep')$ where $f : [0,1] \rightarrow \R$, $X$ is a random variable distributed over $[0,1]$, and $\ep$ and $\ep'$ are (possibly trivial) random variables. We denote the set of all noisy functional relationships by $\F$.
\end{definition}

\subsubsection{Equitability on functional relationships with respect to $R^2$}
We can now state one specific type of equitability on functional relationships: equitability with respect to $R^2$.
\begin{definition}[Equitability on functional relationships with respect to $R^2$]
Let $\Q \subset \F$ be a set of noisy functional relationships. A measure of dependence is {\em $1/d$-equitable} on $\Q$ with respect to $R^2$ if it is $1/d$-interpretable with respect to $R^2$ on $\Q$.
\end{definition}

We observe that this definition still depends on the set $\Q$ in question. The general approach taken in the literature thus far has been to fix some set $F$ of functions that on the one hand is large enough to be representative of relationships encountered in real data sets, but on the other hand is small enough to enable empirical analysis, and to make equitability a realistic goal.

As important as the choice of functions to include in $F$ is the choice of marginal distributions and noise model, both of which are left unspecified in our definition of noisy functional relationships. In past work, we have examined several possibilities. The simplest is $X \sim \mbox{Unif}$, $\ep' \sim \mathcal{N}(0, \sigma^2)$ with $\sigma$ varying, and $\ep = 0$. Slightly more complex noise models include having $\ep$ and $\ep'$ i.i.d. Gaussians, or having $\ep$ be Gaussian and $\ep' = 0$. More complex marginal distributions include having $X$ be distributed in a way that depends on the graph of $f$, or having it be non-stochastic \cite{MINE, reshef2015comparisons}. Given that we often lack a neat description of the noise in real data sets, we would ideally like a statistic to be highly equitable on as many different such models as possible.

We can also easily imagine models besides the ones described above: for instance, we might define $\ep_a$ and $\ep_b$ to be non-Gaussian, we might allow them to depend on each other, or we might allow their variance to depend on $f(X)$. The importance of such modifications depends on the context, but our formalism is designed to be flexible enough to handle general models that include such variations.

\subsubsection{A setting in which perfect equitability is impossible}
One version of equitability on functional relationships for which perfect equitability has been shown to be impossible was introduced by Kinney and Atwal \cite{kinney2014equitability}. This version of equitability uses as standard relationships the set
\[
\Q_K = \left\{ (X, f(X) + \eta) \mbox{ } \big| \mbox{ } f : [0,1] \rightarrow [0,1], (\eta \perp X) | f(X)\right\}
\]
with $\eta$ representing a random variable that is conditionally independent of $X$ given $f(X)$. This model describes functional relationships with noise in the second coordinate only, where that noise can depend arbitrarily on the value of $f(X)$ but must be otherwise independent of $X$.

Kinney and Atwal prove that no non-trivial measure of dependence can be perfectly worst-case interpretable with respect to $R^2$ on the set $\Q_K$. However, we note here that this result, while interesting, has two serious limitations. The first limitation, pointed out by Murrell \textit{et al.} in the technical comment~\cite{Murrell2014comment}, is that $\Q_K$ is extremely large: in particular, the fact that the noise term $\eta$ can depend arbitrarily on the value of $f(X)$ leads to identifiability issues such as obtaining the noiseless relationship $f(X) = X^2$ as a noisy version of $f(X) = X$. The more permissive (i.e. large) a model is, the easier it is to prove an impossibility result for it. Since $\Q_K$ is not contained in the other major models considered in, e.g., \cite{MINE} and \cite{reshef2015comparisons}, it follows that this impossibility result does not imply impossibility for any of those models.

The second limitation of Kinney and Atwal's result is that it only addresses {\em perfect} equitability rather than the more general, approximate notion with which we are primarily concerned.\footnote{
As a matter of record, we wish to clarify a confusion in Kinney and Atwal's work. They write ``The key claim made by Reshef \textit{et al.} in arguing for the use of MIC as a dependence measure has two parts. First, MIC is said to satisfy not just the heuristic notion of equitability, but also the mathematical criterion of $R^2$-equitability...'', with the latter term referring to what we here define as perfect equitability~\cite{kinney2014equitability}. However, such a claim was never made in our previous work \cite{MINE}. Rather, that paper \cite{MINE} informally defined equitability as an approximate notion and compared the equitability of $\MIC$, mutual information estimation, and other schemes empirically, concluding not that $\MIC$ is perfectly equitable but rather that it is the most equitable statistic available in a variety of settings. One method can be more equitable than another, even if neither method is perfectly equitable.
} While a statistic that is perfectly equitable with respect to $R^2$ may indeed be difficult or even impossible to achieve for many large models $\Q$ including some of the models in \cite{MINE} and \cite{reshef2015comparisons}, such impossibility would make {\em approximate} equitability no less desirable a property. The question thus remains how equitable various measures are, both provably and empirically. To borrow an analogy from computer science, the fact that a problem is proven to be NP-complete does not mean that we that we do not want efficient algorithms for the problem; we simply may have to settle for approximate solutions. Similarly, there is merit in searching for measures of dependence that appear to be highly equitable with respect to $R^2$ in practice.

For more on this discussion, see the technical comment~\cite{reshef2014comment}.

\subsection{Quantifying equitability via interpretable intervals}
\label{sec:equitabilityExample}
Let us give a simple demonstration of how the formalism above can be used to empirically quantify equitability with respect to $R^2$ on a specific set of noisy functional relationships. We take as our statistic the sample correlation $\hat \rho$. Since this statistic is meant to detect linear dependencies, we do not expect it to be equitable on a broad class of relationships. In fact it is not even a measure of dependence, since its population value can be zero for relationships with non-trivial dependence. However, we analyze it here as an instructional example since it is widely used and gives intuitive scores. We analyze the equitability of other statistics in Section~\ref{sec:equitabilityAnalysis}.

Figure~\ref{fig:equitabilityExample}a shows an analysis of the equitability with respect to $R^2$ of $\hat \rho$ at a sample size of $n=500$ on the set
\[
\Q = \{ \left( X, f(X) + \ep'_\sigma \right) : X \sim \mbox{Unif}, \ep'_\sigma \sim \mathcal{N}(0, \sigma^2), f \in F, \sigma \in \R_{\geq 0} \}
\]
where $F$ is a set of 16 functions analyzed in \cite{reshef2015comparisons}. (See Appendix~\ref{app:analysis_details}.)

To evaluate the equitability of $\hat \rho$ in this context, we generate, for each function $f \in F$ and for 41 noise levels chosen for each function to correspond to $R^2$ values uniformly spaced in $[0,1]$, $500$ independent samples of size $n=500$ from the relationship $Z_{f, \sigma} = (X, f(X) + \ep'_\sigma)$. We then evaluate $\hat \rho$ on each sample to estimate the 5th and 95th percentiles of the sampling distribution of $\hat \rho$ on $Z_{f, \sigma}$. By taking, for each $\sigma$, the maximal 95th percentile value and the minimal 5th percentile value across all $f \in F$, we obtain estimates of the $0.1$-reliable interval at each noise level. From the reliable intervals we can then construct interpretable intervals, and the equitability of $\hat \rho$ is the reciprocal of the length of the largest interpretable interval.

As expected, the interpretable intervals at many values of $\hat \rho$ are large. This is because our set of functions $F$ contains many non-linear functions, and so a given value of $\hat \rho$ can be assigned to relationships of different types with very different $R^2$ values. This is shown by the pairs of thumbnails in the figure, each of which depicts two relationships with the same $\hat \rho$ but different values of $R^2$. Thus, $\hat \rho$ has poor equitability with respect to $R^2$ on this set $\Q$. In contrast, Figure~\ref{fig:equitabilityExample}b depicts the way this analysis would look if $\rho$ were {\em perfectly} equitable: all the interpretable intervals would have size 0.

\begin{figure}[t]
	\centering
	\begin{tabular}{@{}cc@{}} 
	    \includegraphics[clip=true, trim = 0in 6in 4.32in 2in, width=0.49\textwidth]{\pathToCommonFigs/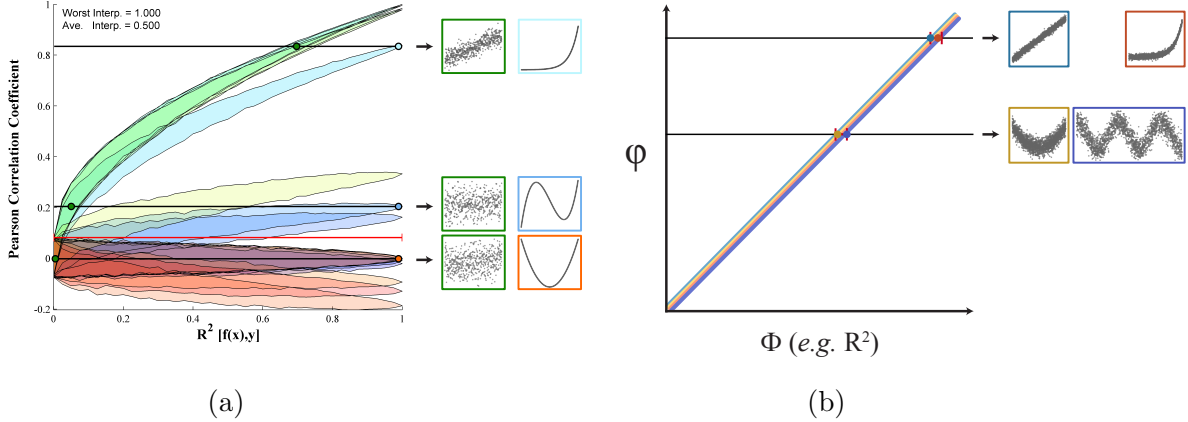}
	    &
	    \includegraphics[clip=true, trim = 4.41in 6in 0in 2in, width=0.48\textwidth]{\pathToCommonFigs/WhatIsEquitability_EquitabilityPaper_v2.pdf} \\
	    (a)\hspace{0.7in} & (b)\hspace{0.8in}
	\end{tabular}
	\caption[Examples of equitable and non-equitable behavior]{
	    Examples of equitable and non-equitable behavior on a set of noisy functional relationships.
    	\figpart{a} The equitability with respect to $R^2$ of the Pearson correlation coefficient $\hat\rho$ over the set $\Q$ of relationships described in Section~\ref{sec:equitabilityExample}, with $n=500$. Each shaded region is an estimated $90\%$ central interval of the sampling distribution of $\hat\rho$ for a given relationship at a given $R^2$. The fact that the interpretable intervals of $\hat\rho$ are large indicates that a given $\hat\rho$ value could correspond to relationships with very different $R^2$ values. This is illustrated by the pairs of thumbnails showing relationships with the same $\hat\rho$ but different $R^2$ values. The largest interpretable interval is indicated by a red line. Because it has width 1, the worst-case equitability with respect to $R^2$ in this case is 1, the lowest possible.
    	\figpart{b} A hypothetical population quantity $\varphi$ that achieves \textit{perfect equitability} in the large-sample limit. Here, the value of $\varphi$ for each relationship type depends only on the $R^2$ of the relationship and increases monotonically with $R^2$. Thus, $\varphi$ can be used as a proxy for $R^2$ on $\Q$ with no loss. Thumbnails are shown for sample relationships that have the same $\varphi$, which corresponds to the fact that they have equal $R^2$ scores. See Appendix~\ref{app:analysis_details} for a legend of the function types used.
	}
	\label{fig:equitabilityExample}
\end{figure}

\subsection{Discussion}
In this section we formalized the notion of equitability via the concepts of reliability and interpretability. Given a statistic $\hvphi$ and a measure of relationship strength $\Phi$ defined on some set $\Q$ of standard relationships, we constructed a set of intervals called the interpretable intervals of $\hvphi$ with respect to $\Phi$. We constructed the interpretable intervals so they yield interval estimates of $\Phi$, and we then defined the (worst-case) equitability of $\hvphi$ to be the inverse of the size of the largest interpretable interval.

Strictly speaking, equitability simply requires that a natural set of confidence intervals obtained from analyzing $\hvphi$ as an estimator of $\Phi$ be small. However, there is a subtlety here: since in our setting $\Q$ typically contains several different relationship types, there are usually multiple relationships in $\Q$ with a given value of $\Phi$. This is different from the conventional framework of estimation of a parameter $\theta$, in which we assume that there is exactly one distribution with any given value of $\theta$, and we must account for this difference in our definitions.

When $\Q$ is so small that this subtlety does not arise, equitability becomes a less rich property. To see this, notice that if there is only one relationship in $\Q$ for every value of $\Phi$, then asymptotic monotonicity of $\hvphi$ with respect to $\Phi$ is sufficient for perfect equitability in the large-sample limit. In this scenario, the main obstacle to the equitability of $\hvphi$ is finite-sample effects, as with parameter estimation. For example, on the set $\Q$ of bivariate Gaussians, many measures of dependence are asymptotically perfectly equitable with respect to the correlation.

However, this differs from the motivating data exploration scenario we consider, in which $\Q$ contains many different relationship types and there are multiple different relationships corresponding to a given value of $\Phi$. Here, equitability can be hindered either by finite-sample effects, or by the differences in the asymptotic behavior of $\hvphi$ on different relationship types in $\Q$. This is illustrated in Figure~\ref{fig:analogy}.

\begin{figure}[t]
	\centering
    \includegraphics[clip=true, trim = 0in 4in 0in 0in, height=0.4\textheight]{\pathToCommonFigs/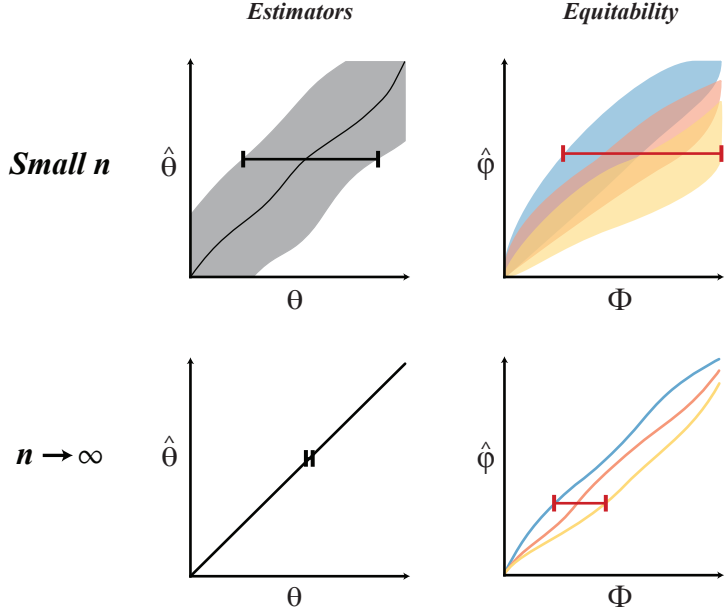}
  \caption{Equitability versus parameter estimation.
The left-hand column depicts a scenario in which $\hat{\theta}$ estimates a parameter $\theta$, each value of which specifies a unique distribution. If the population value of $\hat \theta$ is monotonic in $\theta$, then the confidence intervals shown can be large only due to finite-sample effects.
The right-hand column depicts a scenario in which $\hvphi$ is being used as an estimate of $\Phi$, but a given value of $\Phi$ does not uniquely determine the population value of $\hvphi$: the blue, red, and yellow each represent distinct sets of distributions in $\Q$ whose members can have identical values of $\Phi$. For instance, they might correspond to different function types. This is the setting in which we are operating, and the red intervals on the right are called interpretable intervals. Interpretable intervals can be large either because of finite sample effects (as in the conventional estimation case) or because of the lack of interpretability of the population value of the statistic (shown in the bottom-right picture).} \label{fig:analogy}
\end{figure}

Regardless of the size of $\Q$ though, equitability is fundamentally meant for a situation in which we cannot simply estimate $\Phi$ directly. (In fact, if $\hvphi$ is a consistent estimator of $\Phi$ on $\Q$, it is trivially perfectly equitable in the large-sample limit.) This is because in data exploration we typically require that $\hvphi$ be a measure of dependence in order to obtain a minimal robustness guarantee, and this requirement makes it very difficult to make $\hvphi$ a consistent estimator of $\Phi$ on a large set $\Q$. For instance, suppose $\Q$ is a set of noisy functional relationships and $\Phi = R^2$. Here, computing the sample $R^2$ relative to a non-parametric estimate of the generating function will be asymptotically perfectly equitable. However, this approach is undesirable for data exploration because of its lack of robustness, as exemplified by the fact that it would assign a score of zero to, e.g., a circular relationship. Therefore, we are left with the problem of finding the next-best thing: a measure of dependence $\hvphi$ whose values have a clear, if approximate, interpretation in terms of $\Phi$. Equitability supplies us with a way of talking about how well $\hvphi$ does in this regard.

We close this section with the observation that, though we largely focused here on setting $\Q$ to be some set of noisy functional relationships, the appropriate definitions of $\Q$ and $\Phi$ may change from application to application. For instance, instead of functional relationships one may be interested in relationships supported on one-manifolds, with added noise. Or perhaps instead of $R^2$ one may decide to focus on the mutual information between the sampled y-values and the corresponding de-noised y-values \cite{kinney2014equitability}, or on the fraction of deterministic signal in a mixture \cite{ding2013copula}. In each case the overarching goal should be to have $\Q$ be as large as possible without making it impossible to define an interesting $\Phi$ or making it impossible to find a measure of dependence that achieves good equitability on $\Q$ with respect to this $\Phi$. Finding such families $\Q$ and properties $\Phi$ is an important avenue of future work.

\section{Equitability and statistical power}
\label{sec:equitAndPower}
In the previous section we defined equitability in terms of interval estimation, and observed that the interpretable intervals of a statistic $\hvphi$ with respect to a property of interest $\Phi$ yield interval estimates of $\Phi$ on a set of distributions $\Q$. Given our construction of interpretable intervals via inversion of a set of hypothesis tests, it becomes natural to ask whether there is any connection between equitability and the power of those tests with respect to specific alternatives.

In this section we answer this question by showing that equitability can be equivalently formulated in terms of power with respect to a family of null hypotheses corresponding to different relationship strengths. This result re-casts equitability as a strengthening of power against statistical independence on $\Q$ and gives a second formal definition of equitability that is easily quantifiable using standard power analysis.

Henceforth, we fix the statistic $\hvphi$ and then use $\reliablestat{\alpha}{}{x}$ to denote the $\alpha$-reliable interval of $\hvphi$ at $x \in [0,1]$ and $\interpretablestat{\alpha}{}{x}$ to denote the $\alpha$-interpretable interval of $\hvphi$ at $y \in [0,1]$.

\subsection{Intuition}
Before stating and proving the relationship between equitability and power, let us first build some intuition for why it should hold. We begin by recalling that the reliable interval $\reliablestat{\alpha}{}{x_0}$ is an acceptance region of a two-sided level-$\alpha$ test of $H_0 : \Phi(\mcZ) = x_0$. Since the interval estimates obtained by inverting this test are the interpretable intervals of $\hvphi$, it makes sense to ask whether there is any property of these hypothesis tests that improves as the interpretability of the statistic $\hvphi$ increases. To see why the relevant property is power, let us consider the following illustrative question: what is the minimal $x_1 > 0$ such that a right-tailed\footnote{
We consider a one-sided test here, and henceforth in this section. The reason is because in practice when $\Phi$ corresponds to relationship strength, we are interested in rejecting a null hypothesis representing weaker relationships. In such a situation, it is more common to perform a one-sided test. Nevertheless, results similar to those shown in this section can be derived for two-sided tests as well.
} level-$\alpha$ test of $H_0 : \Phi = 0$ will have power at least $1-\beta$ on $H_1 : \Phi = x_1$? As shown graphically in Figure~\ref{fig:equitabilityAndPowerAgainstIndep}, the answer can be stated in terms of the reliable and interpretable intervals of $\hvphi$.

Specifically, if $t_\alpha$ is the maximal element of $\reliablestat{2\alpha}{}{0}$, then the minimal value of $\Phi$ at which a right-tailed test based on $\hvphi$ will achieve power $1-\beta$ is $\Phi = \max \interpretablestat{2\beta}{}{t_\alpha}$, i.e., the maximal element of the $\beta$-interpretable interval at $t_\alpha$. So if the statistic is highly interpretable at $t_\alpha$, then we will be able to achieve high power against very small departures from the null hypothesis of independence. That is, good interpretability on $\Q$ implies good power against independence on $\Q$. It turns out that this reasoning holds in general and in both directions, as we establish below.

\begin{figure}
	\centering
	\includegraphics[clip=true, trim = 0in 7.6in 3.15in 0in, height=0.2\textheight]{\pathToCommonFigs/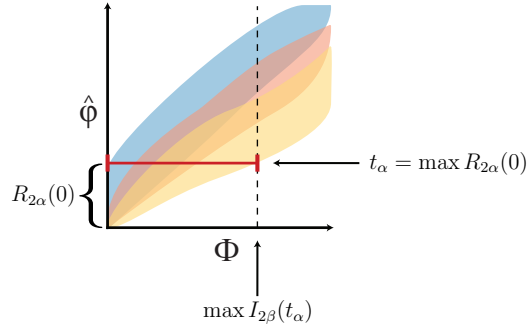}
	\caption[Intuition for the connection between equitability and power]{An illustration of the connection between equitability and power. In this example, we ask for the minimal $x > 0$ that allows a right-tailed test based on $\hvphi$ to achieve power $1-\beta$ in distinguishing between $H_0: \Phi = 0$ and $H_1: \Phi = x$. The optimal critical value of such a test, denoted by $t_\alpha$, can be shown to be the maximal element of the reliable interval $\reliablestat{2\alpha}{}{0}$, and the required $x$ can be shown to be the maximal element of the interpretable interval $\interpretablestat{2\beta}{}{t_\alpha}$, provided $\max \reliablestat{\alpha}{}{\cdot}$ is an increasing function. (The reliable and interpretable intervals pictured are for the case that $\alpha = \beta$.)}
	\label{fig:equitabilityAndPowerAgainstIndep}
\end{figure}

\subsection{Definitions}
To be able to state our main result, we need to formally describe how equitability would be formulated in terms of power. This requires two definitions. The first is a definition of a power function that parametrizes the space of possible alternative hypotheses specifically by the property of interest. The second is a definition of a property of this power function called its uncertain interval. It will turn out later than uncertain intervals are interpretable intervals and vice versa.

As before, let $\hvphi$ be a statistic, let $\Q$ be a set of standard relationships, and let $\Phi : \Q \rightarrow [0,1]$ be a property of interest defined on $\Q$. Given a set of right-tailed tests based on the same test statistic, we refer to the one with the smallest critical value as the most {\em permissive} test.

\begin{definition}
Fix $\alpha, x_0 \in [0,1]$, and let $T^{x_0}_\alpha$ be the most permissive level-$\alpha$ right-tailed test based on $\hvphi$ of the (possibly composite) null hypothesis $H_0 : \Phi(\mcZ) = x_0$. For $x_1 \in [0,1]$, define
\[
\powerfunc{\alpha}{x_0}(x_1) = \inf_{\mcZ : \Phi(\mcZ) = x_1} \Pr{T^{x_0}_\alpha(Z)\mbox{ rejects}}
\]
where $Z$ is a sample of size $n$ from $\mcZ$. That is, $\powerfunc{\alpha}{x_0}(x_1)$ is the power of $T_\alpha^{x_0}$ with respect to the composite alternative hypothesis $H_1 : \Phi = x_1$.

We call the function $\powerfunc{\alpha}{x_0} : [0,1] \rightarrow [0,1]$ the {\em level-$\alpha$ power function} associated to $\hvphi$ at $x_0$ with respect to $\Phi$.
\end{definition}

Note that in the above definition our null and alternative hypotheses may be composite since they are based on $\Phi$ and not on a complete parametrization of $\Q$. That is, $\mcZ$ can be one of several distributions with $\Phi(\mcZ) = x_0$ or $\Phi(\mcZ) = x$ respectively.

Under the assumption that $\Phi(\mcZ) = 0$ if and only if $\mcZ$ represents statistical independence, the power function $\powerfunc{\alpha}{0}$ gives the power of optimal level-$\alpha$ right-tailed tests based on $\hvphi$ at distinguishing various non-zero values of $\Phi$ from statistical independence across the different relationship types in $\Q$. One way to view the main result of this section is that the set of power functions at values of $x_0$ {\em besides} 0 contains much more information than just the power of right-tailed tests based on $\hvphi$ against the null hypothesis of $\Phi = 0$, and that this information can be equivalently viewed in terms of interpretable intervals. Specifically, we can recover the interpretability of $\hvphi$ at every $y \in [0,1]$ by considering its power functions at values of $x_0$ beyond 0.

Let us now define the precise aspect of the power functions associated to $\hvphi$ that will allow us to do this.
\begin{definition}
The {\em uncertain set} of a power function $\powerfunc{\alpha}{x_0}$ is the set $\{x_1 \geq x_0 : \powerfunc{\alpha}{x_0}(x_1) < 1-\alpha \}$.
\end{definition}
The main result of this section will be that uncertain sets are interpretable intervals and vice versa.

\subsection{Preliminary lemmas}
Our proof of the alternate characterization of equitability in terms of power requires two short lemmas. The first shows a connection between the maximum element of a reliable interval and the minimal element of an interpretable interval, namely that these two operations are inverses of each other.

\begin{lemma}
\label{lem:reliable_interpretable_inverse}
Given a statistic $\hvphi$, a property of interest $\Phi$, and some $\alpha \in [0,1]$, define $f(x) = \max \reliablestat{\alpha}{}{x}$ and $g(y) = \min \interpretablestat{\alpha}{}{y}$. If $f$ is strictly increasing, then $f$ and $g$ are inverses of each other.
\end{lemma}
\begin{proof}
Let $y = f(x) = \max \reliablestat{\alpha}{}{x}$. We know that $\min \interpretablestat{\alpha}{}{y} \leq x$, for if it were greater than $x$ then we would have that $x \notin \interpretablestat{\alpha}{}{y}$, which would imply that $y \notin \reliablestat{\alpha}{}{x}$, contradicting the definition of $y$. On the other hand, we cannot have $\min \interpretablestat{\alpha}{}{y} < x$, because this would imply that there is some $x' < x$ such that $y \in \reliablestat{\alpha}{}{x'}$, meaning that $\max \reliablestat{\alpha}{}{x'} \geq y = \max \reliablestat{\alpha}{}{x}$, which contradicts the fact that $f$ is strictly increasing.
\end{proof}

The second lemma gives the connection between reliable intervals and hypothesis testing that we will exploit in our proof.
\begin{lemma}
\label{lem:critical_value}
Fix a statistic $\hvphi$, a property of interest $\Phi$, and some $\alpha, x_0 \in [0,1]$. The most permissive level-$(\alpha/2)$ right-tailed test based on $\hvphi$ of the null hypothesis $H_0 : \Phi(\mcZ) = x_0$ has critical value $\max \reliablestat{\alpha}{}{x_0}$.
\end{lemma}
\begin{proof}
We seek the smallest critical value that yields a level-$(\alpha/2)$ test. This would be the supremum, over all $\mcZ$ with $\Phi(\mcZ) = x_0$, of the $(1-\alpha/2)\cdot 100\%$ value of the sampling distribution of $\hvphi$ when applied to $\mcZ$. By definition this is $\max \reliablestat{\alpha}{}{x_0}$.
\end{proof}

\subsection{Proving the main result: equitability in terms of statistical power}
We are now ready to prove our main result, which is the following equivalent characterization of equitability in terms of statistical power.
\begin{thm}
\label{thm:equivalence}
Fix a set $\Q \subset \P$, a function $\Phi : \Q \rightarrow [0,1]$, and $0 < \alpha < 1/2$. Let $\hvphi$ be a statistic with the property that $\max \reliablestat{2\alpha}{}{x}$ is a strictly increasing function of $x$. Then for all $d > 0$, the following are equivalent.
\begin{enumerate}	
\item $\hvphi$ is worst-case $1/d$-interpretable with respect to $\Phi$ with confidence $1-2\alpha$.
\item For every $x_0, x_1 \in [0,1]$ satisfying $x_1 - x_0 > d$, there exists a level-$\alpha$ right-tailed test based on $\hvphi$ that can distinguish between $H_0 : \Phi(\mcZ) \leq x_0$ and $H_1 : \Phi(\mcZ) \geq x_1$ with power at least $1-\alpha$.
\end{enumerate}
\end{thm}

Theorem~\ref{thm:equivalence} can be seen to follow from the proposition below.

\begin{prop} \label{prop:equitabilityAndPower}
Fix $0 < \alpha < 1$ and $d > 0$, and suppose $\hvphi$ is a statistic with the property that $\max \reliablestat{\alpha}{}{x}$ is a strictly increasing function of $x$. Then for $y \in [0,1]$, the interval $\interpretablestat{\alpha}{}{y}$ equals the closure of the uncertain set of $\powerfunc{\alpha/2}{x_0}$ for $x_0 = \min \interpretablestat{\alpha}{}{y}$. Equivalently, for $x_0 \in [0,1]$, the closure of the uncertain set of $\powerfunc{\alpha/2}{x_0}$ equals $\interpretablestat{\alpha}{}{y}$ for $y = \max \reliablestat{\alpha}{}{x_0}$.
\end{prop}

An illustration of this proposition and its proof is shown in Figure~\ref{fig:equitabilityAndPower}.
\begin{figure}[h!]
	\centering
    \includegraphics[clip=true, trim = 0in 0in 0in 0in, height=0.5\textheight]{\pathToCommonFigs/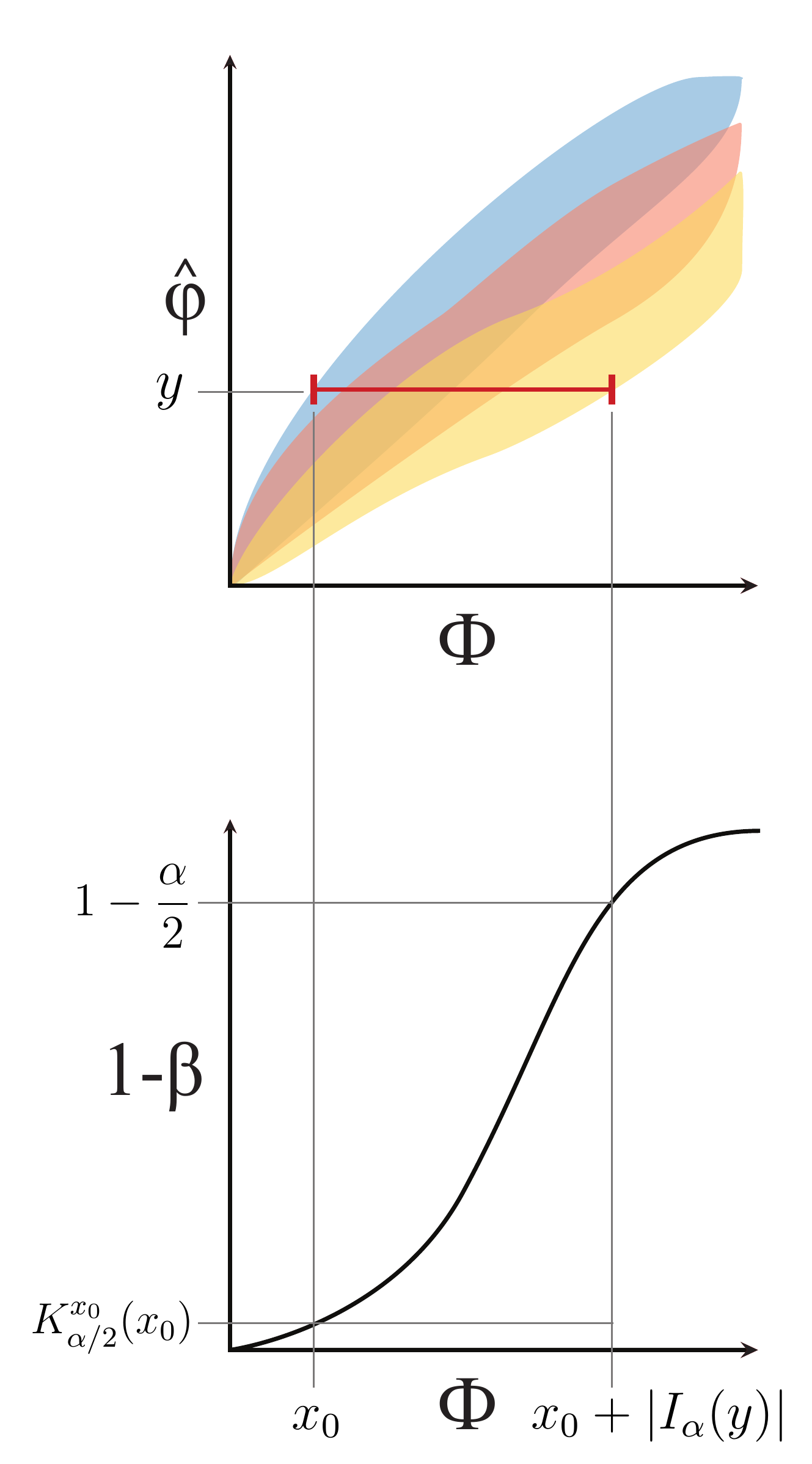}
  \caption{The relationship between equitability and power, as in Proposition~\ref{prop:equitabilityAndPower}.
  The top plot is the same as the one in Figure~\ref{fig:reliabilityInterpretability}a, with the indicated interval denoting the interpretable interval $\interpretablestat{\alpha}{}{y}$. The bottom plot is a plot of the power function $\powerfunc{\alpha/2}{x_0}(x)$, with the y-axis indicating statistical power. The key to the proof of the proposition is to notice that the width of the interpretable interval describes the distance from $x_0$ to the point at which the power function reaches $1-\alpha/2$, and this is exactly the width of the uncertain set of the power function. (Notice that because the null and alternative hypotheses are composite, $\powerfunc{\alpha/2}{x_0}(x_0)$ need not equal $\alpha/2$; in general it may be lower.)}\label{fig:equitabilityAndPower}
\end{figure}

\begin{proof}
The equivalence of the two statements follows from Lemma~\ref{lem:reliable_interpretable_inverse}, which states that $y = \max \reliablestat{\alpha}{}{x_0}$ if and only if $x_0 = \min \interpretablestat{\alpha}{}{y}$. We therefore prove only the first statement, namely that $\interpretablestat{\alpha}{}{y}$ is the uncertain set of $\powerfunc{\alpha/2}{x_0}$ for $x_0 = \min \interpretablestat{\alpha}{}{y}$.

Let $U$ be the uncertain set of $\powerfunc{\alpha/2}{x_0}$. We prove the claim by showing first that $\inf U = \min \interpretablestat{\alpha}{}{y}$, and then that $\sup U = \max \interpretablestat{\alpha}{}{y}$.

To see that $\inf U = \min \interpretablestat{\alpha}{}{y}$, we simply observe that because $\alpha/2 < 1/2$, we have $\powerfunc{\alpha/2}{x_0}(x_0) \leq \alpha/2 < 1-\alpha/2$, which means that $U$ is non-empty, and so by construction its infimum is $x_0$, which we have assumed equals $\min \interpretablestat{\alpha}{}{y}$.

Let us now show that $\sup U \geq \max \interpretablestat{\alpha}{}{y}$: by the definition of the interpretable interval, we can find $x$ arbitrarily close to $\max \interpretablestat{\alpha}{}{y}$ from below such that $y \in \reliablestat{\alpha}{}{x}$. But this means that there exists some $\mcZ$ with $\Phi(\mcZ) = x$ such that if $Z$ is a sample of size $n$ from $\mcZ$ then
\[ \Pr{\hvphi(Z) < y} \geq \frac{\alpha}{2} \]
i.e.,
\[ \Pr{\hvphi(Z) \geq y} < 1- \frac{\alpha}{2} .\]
But since as we already noted $y = \max \reliablestat{\alpha}{}{x_0}$, Lemma~\ref{lem:critical_value} tells us that it is the critical value of the most permissive level-$(\alpha/2)$ right-tailed test of $H_0 : \Phi(\mcZ) = x_0$. Therefore, $\powerfunc{\alpha/2}{x_0}(x) < 1-\alpha/2$, meaning that $x \in U$.

It remains only to show that $\sup U \leq \max \interpretablestat{\alpha}{}{y}$. To do so, we note that $y \notin \reliablestat{\alpha}{}{x}$ for all $x > \max \interpretablestat{\alpha}{}{y}$. This implies that either $y > \max \reliablestat{\alpha}{}{x}$ or $y < \min \reliablestat{\alpha}{}{x}$. However, since $y \in \reliablestat{\alpha}{}{x_0}$ and $\max \reliablestat{\alpha}{}{\cdot}$ is an increasing function, no $x > x_0$ can have $y > \max \reliablestat{\alpha}{}{x}$. Thus the only option remaining is that $y < \min \reliablestat{\alpha}{}{x}$. This means that if $Z$ is a sample of size $n$ from any $\mcZ$ with $\Phi(\mcZ) = x > \max \interpretablestat{\alpha}{}{y}$, then
\[ \Pr{\hvphi(Z) < y} < \frac{\alpha}{2} \]
i.e.,
\[ \Pr{\hvphi(Z) \geq y} \geq 1- \frac{\alpha}{2} .\]
As above, this implies that $\powerfunc{\alpha/2}{x_0}(x) \geq 1-\alpha/2$, which means that $x \notin U$, as desired.
\end{proof}

\subsection{Quantifying equitability via statistical power}
\label{sec:equitabilityViaPower}
Theorem~\ref{thm:equivalence} gives us an alternative to measuring equitability via lengths of interpretable intervals. Instead, for every $x_0 \in [0,1)$ and for every $x_1 > x_0$, we can use many samples of size $n$ to estimate the power of right-tailed tests based on $\hvphi$ at distinguishing $H_0: \Phi = x_0$ from $H_1 :\Phi = x_1$. This process is illustrated schematically in Figure~\ref{fig:equitability_power_schematic}. In that figure, good equitability corresponds to high power on pairs $(x_1, x_0)$ even when $x_1 - x_0$ is small.

\begin{figure}[h!]
	\centering
    \includegraphics[clip=true, trim = 0in 1.1in 0in 3.18in, width=0.8\textwidth]{\pathToCommonFigs/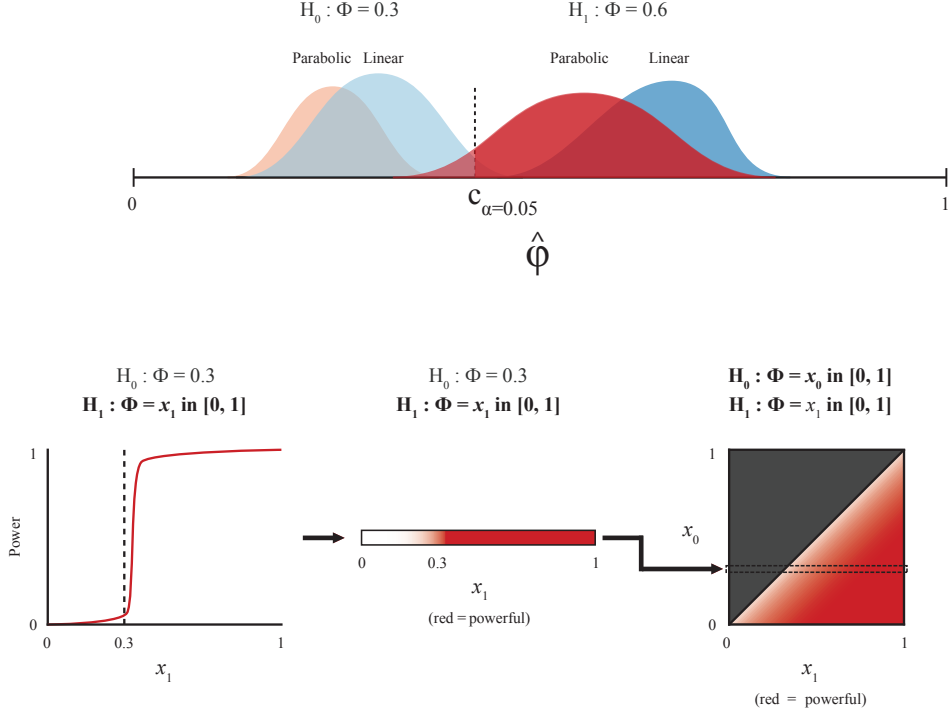}
  \caption{A schematic illustration of the visualization of equitability via statistical power.
  \figpart{Top} A depiction of the sampling distributions of a test statistic $\hvphi$ when a data set contains only four relationships: a parabolic and a linear relationship with $\Phi = 0.3$, and a parabolic and a linear relationship with $\Phi = 0.6$. The dashed line represents the critical value of the most permissive level-$\alpha$ right-tailed test of $H_0 : \Phi = 0.3$.
  \figpart{Bottom left} The power function of the most permissive level-$\alpha$ right-tailed test based on a statistic $\hvphi$ of the null hypothesis $H_0: \Phi = 0.3$. The curve shows the power of the test as a function of $x_1$, the value of $\Phi$ that defines the alternative hypothesis.
  \figpart{Bottom middle} The power function can be depicted instead as a heat map. 
  \figpart{Bottom right} Instead of considering just one null hypothesis, we can consider a set of null hypotheses (with corresponding critical values) of the form $H_0: \Phi = x_0$ and plot each of the resulting power curves as a heat map. The result is a plot in which the intensity of the color in the coordinate $(x_1, x_0)$ corresponds to the power of the size-$\alpha$ right-tailed test based on $\hvphi$ at distinguishing $H_1 : \Phi = x_1$ from $H_0 : \Phi = x_0$. A statistic is $1/d$-equitable with confidence $1-2\alpha$ if this power surface attains the value $1-\alpha$ within distance $d$ of the diagonal along each row. In other words, the redder the triangle appears, the higher the equitability of $\hvphi$.
  }
  \label{fig:equitability_power_schematic}
\end{figure}

\subsection{Discussion}
In this section, we gave a characterization of equitability in terms of statistical power with respect to a family of null hypotheses corresponding to different relationship strengths. (See Theorem~\ref{thm:equivalence}.) This characterization shows what the concept of equitability/interpretability is fundamentally about: being able to distinguish not just signal ($\Phi > 0$) from no signal ($\Phi = 0$) but also stronger signal ($\Phi = x_1$) from weaker signal ($\Phi = x_0$), and being able to do so across relationships of different types. This indeed makes sense when a data set contains an overwhelming number of heterogeneous relationships that exhibit, say, $\Phi(\mcZ) = 0.3$ and that we would like to ignore because they are not as interesting as the small number of relationships with, say, $\Phi(\mcZ) = 0.8$.

Let us now explore how the power requirement into which equitability translates differs from the conventional lens through which measures of dependence are analyzed. We do so by returning once more to the case in which $\Q$ is a set of noisy functional relationships and the property of interest is $R^2$. In this setting, the conventional way to assess a measure of dependence would be through analysis of its power with respect to a null hypothesis of independence and with a simple alternative hypothesis. Such an analysis would consider, say, right-tailed tests based on the statistic $\hvphi$ and evaluate their power at rejecting the null hypothesis of $R^2 = 0$, i.e. statistical independence, first on linear relationships with varying noise levels, then separately on exponential relationships with varying noise levels, and so on.

In contrast, our result shows that for $\hvphi$ to be $1/d$-equitable, it must yield right-tailed tests with high power at distinguishing null hypotheses of the form $R^2 \leq x_0$ from alternative hypotheses of the form $R^2 \geq x_1$ for {\em any} $x_1 > x_0 + d$. This is more stringent than the conventional analysis described above for the following three reasons.
\begin{enumerate}
\item Instead of just one null hypothesis $x_0$ (i.e., $x_0 = 0$), there are many possible values of $x_0$ corresponding to different $R^2$ values.
\item Each of the new null hypotheses can be composite since $\Q$ can contain relationships of many different types (e.g. noisy linear, noisy sinusoidal, and noisy parabolic). Whereas for many measures of dependence all of these relationships may have reduced to a single null hypothesis of statistical independence in the case of $R^2 = 0$, they yield composite null hypotheses once we allow $R^2$ to be non-zero.
\item The alternative hypotheses here are also composite, since each one similarly consists of several different relationship types with the same $R^2$. Whereas conventional analysis of power against independence considers only one alternative at a time, here we require that tests simultaneously have good power on sets of alternatives with the same $R^2$.
\end{enumerate}

This understanding of equitability is both good news and bad news. On the one hand, it provides us with a concrete sense of the relationship of equitability to power against independence, which has been the more traditional way of evaluating measures of dependence. In so doing, it also makes clear the motivation behind equitability and the cases in which it is useful. On the other hand, however, the understanding that equitability corresponds to power against a much larger set of null hypotheses suggests, via ``no free lunch''-type considerations, that if we want to achieve higher power against this larger set of null hypotheses, we may need to give up some power against independence. And indeed, in \cite{reshef2015comparisons} we demonstrate empirically that such a trade-off does seem to exist for several measures of dependence.

%
However, there are situations in which it may be desirable to give up some power against independence in exchange for a degree of equitability. For instance, recall the analysis \cite{heller2014consistent} of the gene expression data set discussed earlier in this paper. In that analysis, not only did several measures of dependence each detect thousands of significant relationships after correction for multiple hypothesis testing, but there was also an overlap of over $85\%$ among the relationships detected by the five best-performing methods. In data exploration scenarios such as this one, in which existing measures of dependence reliably identify so many relationships, focusing on additional gains in power against independence appears less of a significant priority than deciding how to choose among the large number of relationships already detected.

\section{Equitability implies low detection threshold}
The primary motivation given for equitability is that often data sets contain so many relationships that we are not interested in all deviations from independence but rather only in the strongest few relationships. However, there are also many data sets in which, due to low sample size, multiple-testing considerations, or relative lack of structure in the data, very few relationships pass significance. Alternatively, there are also settings in which equitability is too ambitious even at large sample sizes. In such settings, we may indeed be interested in simply detecting deviations from independence rather than ranking them by strength.

In this situation, there is still cause for concern about the effect on our results of our choice of test statistic $\hvphi$. For instance, it is easy to imagine that, despite asymptotic guarantees, an independence test will suffer from low power even on strong relationships of a certain type at a finite sample size $n$ because the test statistic systematically assigns lower scores to relationships of that type. To avoid this, we might want a guarantee that, at a sample size of $n$, the test has a given amount of power in detecting relationships whose strength as measured by $\Phi$ is above a certain threshold, across a broad range of relationship types. This would ensure that, even if we cannot rank relationships by strength, we at least will not miss important relationships as a result of the statistic we use.

In this section we show a straightforward connection between equitability as defined above and this desideratum, which we call {\em low detection threshold}. In particular, we show via the alternate characterization of equitability proven in the previous section that low detection threshold is a straightforward consequence of high equitability. Since the converse does not hold, low detection threshold may be a reasonable criterion to use in situations in which equitability is too much to ask.

Given a set $\Q$ of standard relationships, and a property of interest $\Phi$, we define low detection threshold as follows.
\begin{definition}
A statistic $\hvphi$ has a {\em $(1-\beta)$-detection threshold} of $d$ at level $\alpha$ with respect to $\Phi$ on $\Q$ if there exists a level-$\alpha$ right-tailed test based on $\hvphi$ of the null hypothesis $H_0 : \Phi(\mcZ) = 0$ whose power on $H_1 : \mcZ$ at a sample size of $n$ is at least $1-\beta$ for all $\mcZ \in \Q$ with $\Phi(\mcZ) > d$.
\end{definition}

The connection between equitability and low detection threshold is then a straightforward corollary of Theorem~\ref{thm:equivalence}.
\begin{cor}
Fix some $0 < \alpha < 1$, let $\hvphi$ be worst-case $1/d$-interpretable with respect to $\Phi$ on $\Q$ with confidence $1-2\alpha$, and assume that $\max \reliablestat{2\alpha}{}{\cdot}$ is a strictly increasing function. Then $\hvphi$ has a $(1-\alpha)$-detection threshold of $d$ at level $\alpha$ with respect to $\Phi$ on $\Q$.
\end{cor}

Assume that $\Phi$ has the property that it is zero precisely in cases of statistical independence. Then the above corollary says that equitability and interpretability --- to the extent they can be achieved --- make strong guarantees about power against independence on $\Q$. On the other hand, it is easy to see that low detection threshold need not imply equitability. Therefore, minimal power against independence is a strictly weaker criterion than equitability.

The connection between equitability and detection threshold with respect to $\Phi$ is important because there exist situations in which equitability may be difficult to achieve but in which we still want some sort of guarantee about the robustness of our power against independence to changes in relationship type. This general theme of not missing relationships because of their type is the intuitive heart of equitability, and the above corollary shows how this conception might be utilized in other ways.

Another way that low detection threshold arises naturally is if we pre-filter our data set using some independence test before conducting a more fine-grained analysis with a second statistic. In that case, low detection threshold ensures that we will not ``throw out'' important relationships prematurely just because of their relationship type. In our companion paper \cite{reshef2015comparisons}, we propose precisely such a scheme, and we analyze the detection threshold of the preliminary test in question to argue that the scheme will perform well.

\section{Quantifying equitability in practice}
\label{sec:equitabilityAnalysis}
Having defined equitability and seen how it can be interpreted in terms of power, we now consider the equitability on a set of noisy functional relationships of some commonly used methods: the maximal information coefficient as estimated by $\MICestE$ \cite{reshef2015estimating}, distance correlation \cite{szekely2007measuring, szekely2009brownian, huo2014fast}, and mutual information \cite{Cover2006, csiszar2008axiomatic} as estimated using the Kraskov estimator \cite{Kraskov}.

In this analysis, we use $\Phi = R^2$ as our property of interest, $n=500$ as our sample size, and
\[
\Q = \{ (x + \ep_\sigma, f(x) + \ep'_\sigma) : x \in X_f, \ep_\sigma, \ep'_\sigma \sim \mathcal{N}(0, \sigma^2), f \in F, \sigma \in \R_{\geq 0} \}
\]
where $\ep_\sigma$ and $\ep'_\sigma$ are i.i.d., $F$ is the set of functions in Appendix~\ref{app:analysis_details}, and $X_f$ is the set of $n$ x-values that result in the points $(x_i, f(x_i))$ being equally spaced along the graph of $f$.

The results of the analysis are shown in Figure~\ref{fig:equitabilityAnalysis}. The figure visualizes the analysis via both interpretable intervals and statistical power. By Theorem~\ref{thm:equivalence}, these two viewpoints are equivalent, and they are both shown here in order to help the reader build intuition for this equivalence. For instance, the worst-case $0.1$-interpretability of $\MICestE$ here is $2.92$, because the widest interpretable interval is of size $2.92$. And indeed, $\MICestE$ yields right-tailed tests with $1-0.1/2 = 95\%$ power at distinguishing any null hypothesis of the form $H_0: R^2(\mcZ) = x_0$ from any alternative hypothesis of the form $H_1: R^2(\mcZ) = x_1$ provided $x_1 - x_0 > 1/2.92 = 0.342$.

As the figure demonstrates, the equitability of $2.92$ achieved by $\MICestE$ on this $\Q$ is the highest among the methods examined. In contrast, the equitabilities with respect to $R^2$ of distance correlation and mutual information estimation on this $\Q$ are $1$ and $1.04$, respectively. For a more extensive analysis that varies the sample size as well as noise model and marginal distributions, and compares many more methods, see \cite{reshef2015comparisons}.

\begin{figure}[t]
\centering
    \includegraphics[clip=true, trim = 0.5in 0in 0.1in 0in,width=0.9\textwidth]{\pathToFigures/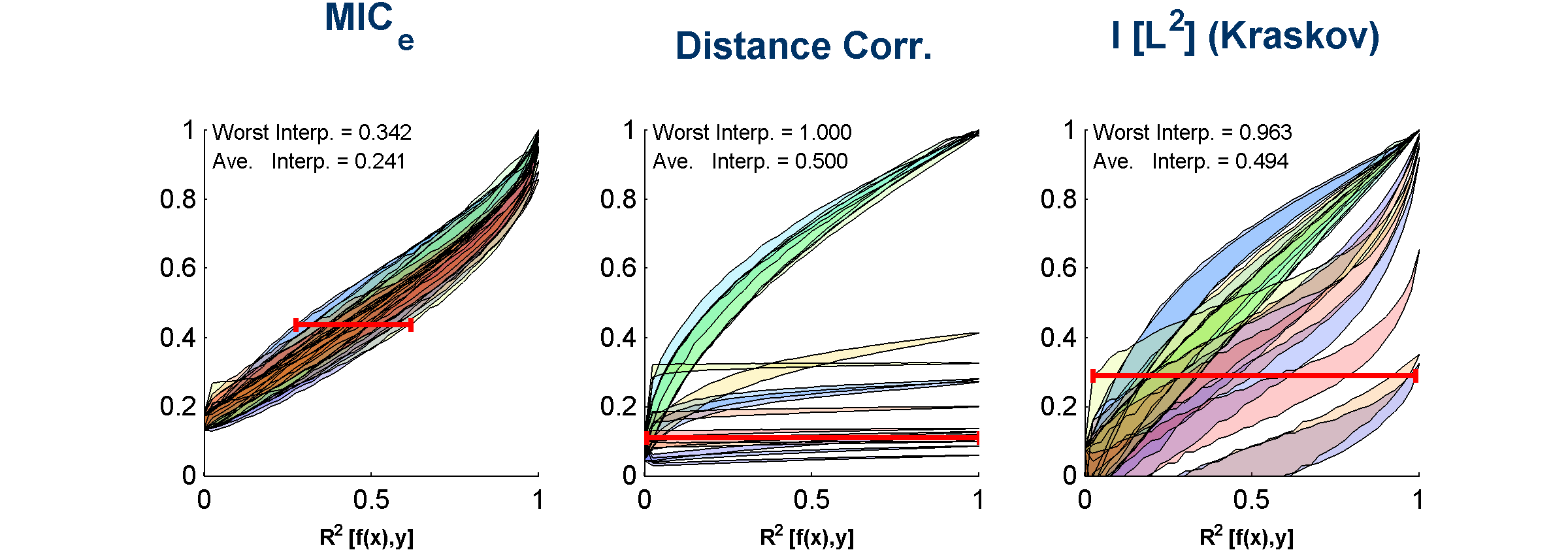}
    \\
    \vspace{1\bigskipamount}
    \includegraphics[clip=true, trim = 0.5in 0in 0.1in 0in,width=0.9\textwidth]{\pathToFigures/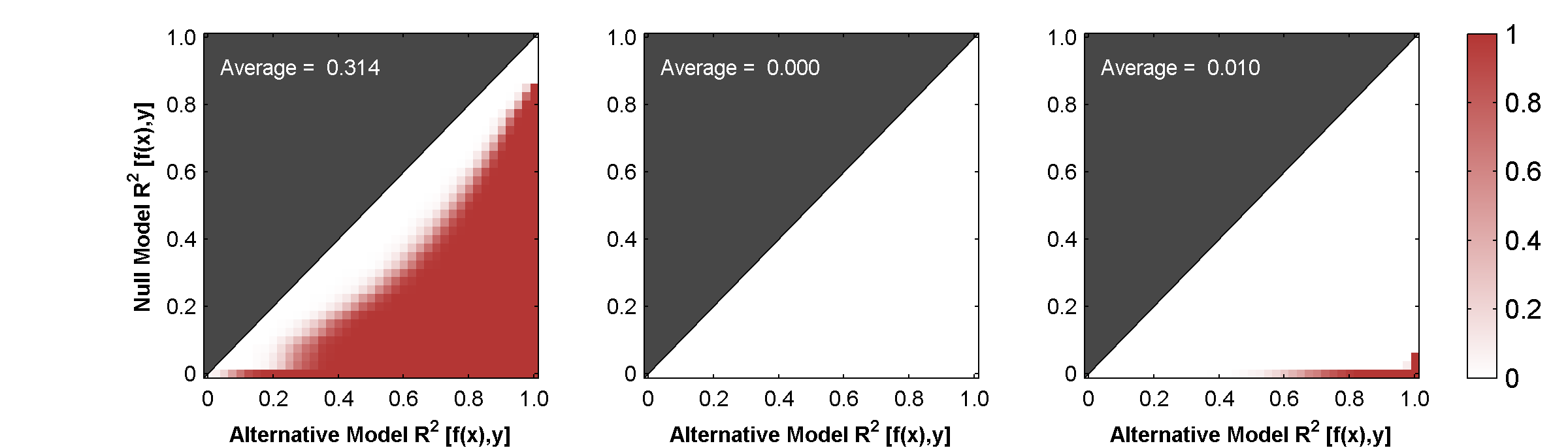}
\caption{An analysis of the equitability with respect to $R^2$ of three measures of dependence on a set of functional relationships.
    The set of relationships used is described in Section~\ref{sec:equitabilityAnalysis}. Each column contains results for the indicated measure of dependence.
    \figpart{Top} The analysis visualized via interpretable intervals as in Figure~\ref{fig:equitabilityExample}. \textit{[Narrower is more equitable.]} The worst-case and average-case widths of the $0.1$-interpretable intervals for the statistic in question are indicated.
    \figpart{Bottom} The same analysis visualized via statistical power as in Figure~\ref{fig:equitability_power_schematic}. \textit{[Redder is more equitable.]} The average power across all pairs of null and alternative hypotheses is computed for each plot.
    For a legend describing which functional relationships were analyzed and which parameters were used for each method, see Appendix~\ref{app:analysis_details}.}
\label{fig:equitabilityAnalysis}
\end{figure}

\section{Conclusion}
\label{sec:conclusion}
Informally, given some measure $\Phi$ of relationship strength, the {\em equitability} of a measure of dependence $\hvphi$ with respect to $\Phi$ is the degree to which $\hvphi$ allows us to draw inferences about relationship strength across a broad set of relationship types. We give here a conceptual framework to motivate equitability and then discuss the contributions of this work.

\subsubsection{The motivation for equitability}
There are two different ways to motivate equitability. The first is to begin with a measure of dependence $\hvphi$ and to observe that, though $\hvphi$ will asymptotically allow us to detect all deviations from independence in a data set, it need not tell us anything about the strength of those relationships. Since it often happens that we detect many more relationships than can be realistically followed up, it would be desirable to have $\hvphi$ tell us something not just about the presence or absence of a relationship, but also about relationship strength as defined by $\Phi$ on at least a partial set of ``standard relationships'' $\Q$.

The second way is to suppose that $\hvphi$ is a consistent estimator of $\Phi$ on $\Q$ and to ask ``what is the minimal requirement we can add to ensure that $\hvphi$ is robust to detecting relationships outside of $\Q$?'' Perhaps the weakest stipulation we can impose is that the population value $\varphi$ of our statistic be non-zero in cases of non-trivial dependence of any sort. That is, we want $\hvphi$ to be a measure of dependence as well.

Both of these scenarios would be resolved by a measure of dependence that is also a consistent estimator of $\Phi$. However, in many interesting cases there is no known statistic satisfying both properties: for instance, if $\Q$ is a set of noisy functional relationships and $\Phi$ is $R^2$, then on the one hand computing the sample $R^2$ with respect to a non-parametric estimate of the generating function will be a consistent estimator of $\Phi$, but will give a score of 0 to a circle. And on the other hand, no measure of dependence is known also to be a consistent estimator of $R^2$ on noisy functional relationships.

This naturally leads us to wonder whether, despite the difficulty of simultaneously estimating $\Phi$ consistently and retaining the properties of a measure of dependence, we can at least seek an approximate version of this ideal. Doing so, however, requires a weaker requirement than consistent estimation. This is what leads us to equitability. Equitability allows us to seek statistics that have the robustness of measures of dependence but that also, via their relationship to a property of interest $\Phi$, give values that have a clear, if approximate, interpretation and can therefore be used to rank relationships.

\subsubsection{Contributions of this work}
In this paper, we formalized and developed the theory of equitability in three ways. We first defined the equitability of a statistic $\hvphi$ on $\Q$ with respect to $\Phi$ as the extent to which $\hvphi$ give us good interval estimates of $\Phi$ on $\Q$. Our definition rests on an object called the interpretable interval, which has coverage guarantees with respect to $\Phi$. We define $\hvphi$ to be equitable if all of its interpretable intervals are small.

Second, we showed that this formalization of equitability can be equivalently stated in terms of power against a specific set of null hypotheses corresponding to different relationship strengths. That is, while measures of dependence have conventionally been judged by their power at distinguishing non-trivial signal from statistical independence, equitability is equivalent to the stronger property of being able to distinguish different degrees of possibly non-trivial signal strength from each other.

Third, we defined a concept called low detection threshold, which stipulates that, at a fixed sample size, a statistic yield independence tests with a guaranteed minimal power to detect relationships whose strength passes a certain threshold, across a range of relationship types. We showed that low detection threshold is a straightforward consequence of equitability. Since the converse does not hold, low detection threshold is a natural weaker criterion that one could aim for when equitability proves difficult to achieve.

Our formalization and its results serve three primary purposes. The first is to provide a framework for rigorous discussion and exploration of equitability and related concepts. The second is to situate equitability in the context of interval estimation and hypothesis testing and to clarify its relationship to central concepts in those areas such as confidence and statistical power. The third is to show that equitability and the language developed around it can help us to both formulate and achieve other useful desiderata for measures of dependence.

These connections provide a framework for thinking about the utility of both current and future measure of dependence for exploratory data analysis. Power against independence, the lens through which measures of dependence are currently evaluated, is appropriate in many settings in which very few significant relationships are expected, or in which we want to know whether one specific relationship is non-trivial or not. However, in situations in which most measures of dependence already identify a large number of relationships, a rigorous theory of equitability will allow us to begin to assess when we can glean more information from a given measure of dependence than just the binary result of an independence test.

Of course, there is much left to understand about equitability. For instance, to what extent is it achievable for different properties of interest? What are natural and useful properties of interest for sets $\Q$ besides noisy functional relationships? For common statistics such as $\MIC$ \cite{MINE} or $\MICestE$ \cite{reshef2015estimating}, can we obtain a theoretical characterization of the sets $\Q$ for which good equitability with respect to $R^2$ is achieved? Are there systematic ways of obtaining equitable behavior via a learning framework as was done for causation in \cite{lopez2015towards}? These questions all deserve attention.

Equitability as framed here is certainly not the only goal to which we should strive in developing new measures of dependence. As data sets not only grow in size but also become more varied, there will undoubtedly develop new and interesting use-cases for measures of dependence, each with its own way of assessing success. Notwithstanding which particular modes of assessment are used, it is important that we formulate and explore concepts that move beyond power against independence, at least in the bivariate setting. Equitability provides one approach to coping with the changing nature of data exploration, but more generally, we can and should ask more of measures of dependence.

\section{Acknowledgments}
The authors would like to acknowledge R Adams, E Airoldi, H Finucane, A Gelman, M Gorfine, R Heller, J Huggins, J Mueller, and R Tibshirani for constructive conversations and useful feedback.

\bibliographystyle{ieeetr}
\bibliography{\pathToCommon/References}

\newpage

\appendix

\section{Details of analyses}
\label{app:analysis_details}
\subsection{Functions analysed in Figures~\ref{fig:equitabilityExample} and~\ref{fig:equitabilityAnalysis}}
Below is the legend showing which function types correspond to the colors in each of Figures~\ref{fig:equitabilityExample} and~\ref{fig:equitabilityAnalysis}. The functions used are the same as the ones in the equitability analyses of \cite{reshef2015comparisons}.

\begin{figure}[h!]
\centering
\includegraphics[clip=true, trim = 4.125in 1.675in 1.75in 1.675in,width=0.3\textwidth]{\pathToFigures/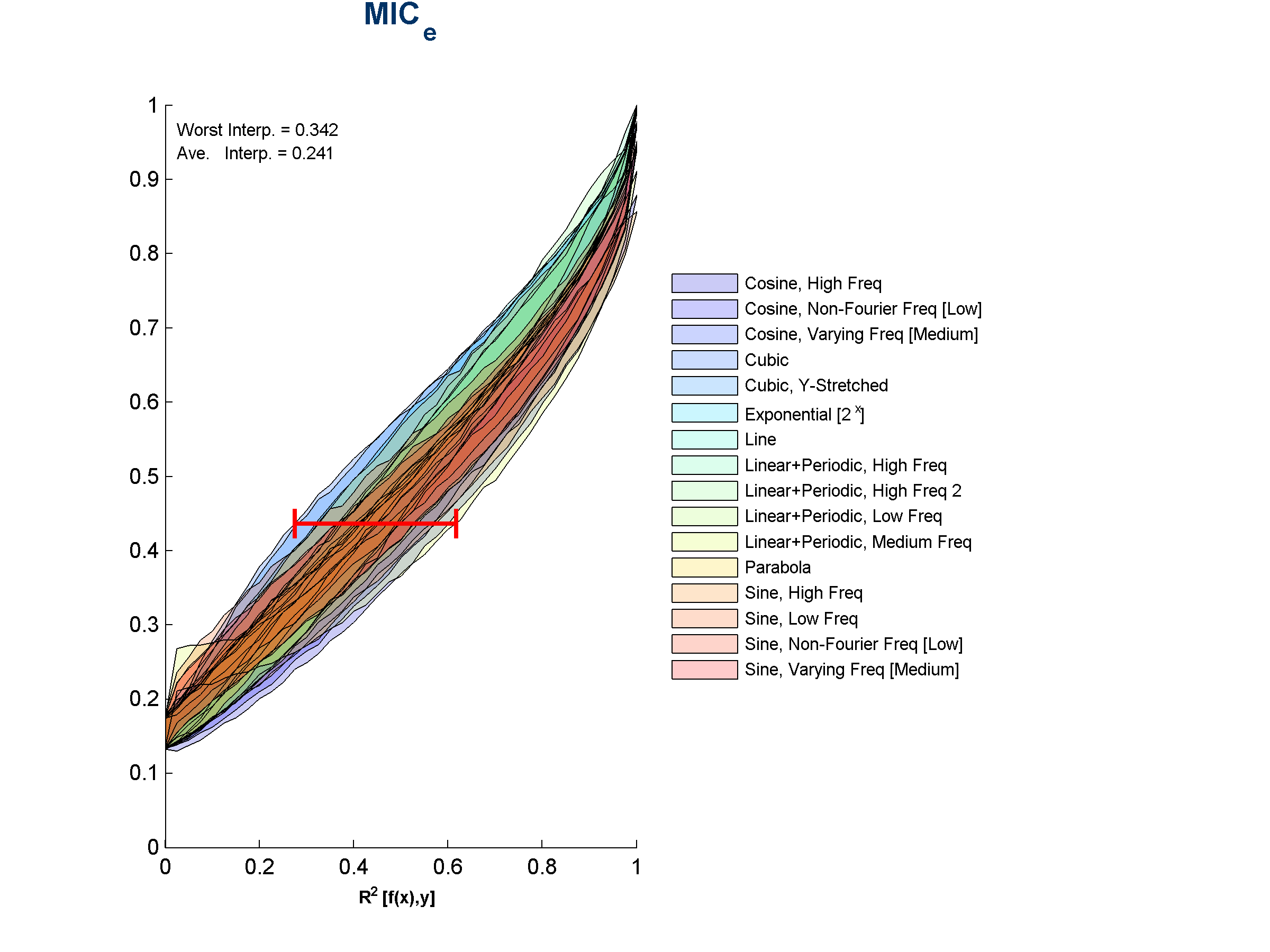}
\caption*{The legend for Figures~\ref{fig:equitabilityExample} and~\ref{fig:equitabilityAnalysis}.}
\label{fig:my_label}
\end{figure}

\subsection{Parameters used in Figure~\ref{fig:equitabilityAnalysis}}
In the analysis of the equitability of $\MICestE$, distance correlation, and mutual information, the following parameter choices were made: for $\MICestE$, $\alpha = 0.8$ and $c=5$ were used; for distance correlation no parameter is required; and for mutual information estimation via the Kraskov estimator, $k=6$ was used. The parameters chosen were the ones that maximize overall equitability in the detailed analyses performed in \cite{reshef2015comparisons}. For mutual information, the choice of $k=6$ (out of the parameters tested: $k=1,6,10,20$) also maximizes equitability on the specific set $\Q$ that is analyzed in Figure~\ref{fig:equitabilityAnalysis}.

\end{document}

%% file: commands.tex
\newcommand{\R}{\mathbb{R}}

\newcommand{\ep}{\varepsilon}
\renewcommand{\Pr}[1]{\mathbf{P} \left( {#1} \right)}

\renewcommand{\P}{\boldsymbol{\mathcal{P}}}
\newcommand{\Q}{\boldsymbol{\mathcal{Q}}}
\newcommand{\F}{\boldsymbol{\mathcal{F}}}
\newcommand{\mcZ}{\mathcal{Z}}
\newcommand{\hvphi}{{\hat{\varphi}}}
\newcommand{\reliable}[2]{ {R^{#1}\left( #2 \right)} }
\newcommand{\reliablestat}[3]{ {R^{#2}_{#1}\left( #3 \right)} }
\newcommand{\interpretable}[2]{ {I^{#1}\left( #2 \right)} }
\newcommand{\interpretablestat}[3]{ {I^{#2}_{#1}\left( #3 \right)} }
\newcommand{\powerfunc}[2]{ {K^{#2}_{#1}} }

\newcommand{\MIC}{\mbox{MIC}}

\newcommand{\MICestE}{{\mbox{MIC}_e}}

\newtheorem{thm}{Theorem}[section]
\newtheorem*{thm*}{Theorem}
\newtheorem{prop}[thm]{Proposition}
\newtheorem{lemma}[thm]{Lemma}
\newtheorem{cor}[thm]{Corollary}

\theoremstyle{definition}
\newtheorem{definition}[thm]{Definition}

\newcommand{\figpart}[1]{\textbf{({#1})}}